\documentclass[a4paper,12pt]{article}

\usepackage{amscd,amssymb,amsmath,amsfonts,amsthm, mathrsfs}
\usepackage{stmaryrd, accents}
\usepackage{graphicx}
\usepackage{verbatim,enumerate,paralist}
\usepackage{textcomp}
\usepackage{tensor}
\usepackage[colorlinks]{hyperref}
\hypersetup{%
  urlcolor=blue,linkcolor=blue,citecolor=blue
}
\usepackage{pdfpages}

\usepackage{tikz-cd}

\usepackage[T1]{fontenc}
\usepackage[scr=rsfs,cal=boondox]{mathalfa}

\usepackage[arrow, matrix, tips, curve, graph, rotate]{xy}
\SelectTips{cm}{10}

\makeatletter
\def\matrixobject@{%
  \edef \next@{={\DirectionfromtheDirection@ }}%
  \expandafter \toks@ \next@ \plainxy@
  \let\xy@@ix@=\xyq@@toksix@
  \xyFN@ \OBJECT@}
\let\xy@entry@@norm=\entry@@norm
\def\entry@@norm@patched{%
  \let\object@=\matrixobject@
  \xy@entry@@norm }
\AtBeginDocument{\let\entry@@norm\entry@@norm@patched}
\makeatother

\newcommand{\twocong}[2][0.5]{\ar@{}[#2] \save ?(#1)*{\cong}\restore}
\newcommand{\twoeq}[2][0.5]{\ar@{}[#2] \save ?(#1)*{=}\restore}
\newcommand{\ltwocell}[3][0.5]{\ar@{}[#2] \ar@{=>}?(#1)+/r 0.2cm/;?(#1)+/l 0.2cm/^{#3}}
\newcommand{\rtwocell}[3][0.5]{\ar@{}[#2] \ar@{=>}?(#1)+/l 0.2cm/;?(#1)+/r 0.2cm/^{#3}}
\newcommand{\utwocell}[3][0.5]{\ar@{}[#2] \ar@{=>}?(#1)+/d  0.2cm/;?(#1)+/u 0.2cm/_{#3}}
\newcommand{\dtwocell}[3][0.5]{\ar@{}[#2] \ar@{=>}?(#1)+/u  0.2cm/;?(#1)+/d 0.2cm/^{#3}}
\newcommand{\ultwocell}[3][0.5]{\ar@{}[#2] \ar@{=>}?(#1)+/dr  0.2cm/;?(#1)+/ul 0.2cm/^{#3}}
\newcommand{\urtwocell}[3][0.5]{\ar@{}[#2] \ar@{=>}?(#1)+/dl  0.2cm/;?(#1)+/ur 0.2cm/^{#3}}
\newcommand{\dltwocell}[3][0.5]{\ar@{}[#2] \ar@{=>}?(#1)+/ur  0.2cm/;?(#1)+/dl 0.2cm/^{#3}}
\newcommand{\drtwocell}[3][0.5]{\ar@{}[#2] \ar@{=>}?(#1)+/ul  0.2cm/;?(#1)+/dr 0.2cm/^{#3}}

\makeatletter
\newcommand{\xRightarrow}[2][]{\ext@arrow 0359\Rightarrowfill@{#1}{#2}}
\makeatother

\newtheorem{theorem}{Theorem}[section]
\newtheorem{corollary}[theorem]{Corollary}
\newtheorem{lemma}[theorem]{Lemma}
\newtheorem{proposition}[theorem]{Proposition}

\newtheorem{definition}[theorem]{Definition}
\newtheorem{example}[theorem]{Example}

\newtheoremstyle{step}{2\bigskipamount}{\medskipamount}{\upshape}{}{\itshape}{. }{ }{\underline{Step~\thestep}}
\theoremstyle{step}

\renewcommand{\thestep}{\arabic{step}}

\newcommand{\Lra}{\Longrightarrow}
\newcommand{\Ra}{\Rightarrow}

\newcommand{\ldual}[1]{\mathord{{\let\nolimits\relax\sideset{^\wedge}{}{#1}}}}
\newcommand{\laction}[2]{\mathord{{\let\nolimits\relax\sideset{^{#1}}{}{#2}}}}
\newcommand{\conj}[2]{\mathord{{\let\nolimits\relax\sideset{^{#1}}{}{#2}}}}

\newcommand{\xra}{\xrightarrow}
\newcommand{\xla}{\xleftarrow}
\makeatletter
\newcommand{\xRa}[2][]{\ext@arrow 0359\Rightarrowfill@{#1}{#2}}
\makeatother

\usepackage[utf8]{inputenc}
\DeclareFontFamily{U}{min}{}
\DeclareFontShape{U}{min}{m}{n}{<-> udmj30}{}
\newcommand\yo{\!\text{\usefont{U}{min}{m}{n}\symbol{'207}}\!}

\def\CA{{\mathscr A}}

\def\CD{{\mathscr D}}

\def\CH{{\mathscr H}}

\def\CK{{\mathscr K}}

\def\CM{{\mathscr M}}

\def\CV{{\mathscr V}}
\def\CW{{\mathscr W}}
\def\CX{{\mathscr X}}

\def\ot{{\otimes}}

\DeclareMathAlphabet{\mathbbe}{U}{bbold}{m}{n}

\begin{document}

\author{Ross Street}

\title{Homodular pseudofunctors and bicategories of modules} 
\date{\today}
\maketitle
\\
{\small{\emph{Key words and phrases:} enriched category; bicategory; distributor; module; fibred category; homology; cofibration; coopfibration; comma category; coslice}

%\tableofcontents
%\newpage
%\begin{abstract}
%\noindent  
\section*{Introduction}

The universal property for the B\'enabou bicategory of distributors 
(although we  call them ``modules'') presented here is somewhat implicitly spread over the series of papers
\cite{BungeThesis, LawMetric, Ben1973, G, GG, 8, 14, 85, 22, 46} and yet, to my knowledge, 
does not appear in print.
The inclusion of a bicategory $\CW$ into the bicategory $\CW\text{-}\mathrm{Mod}$
of $\CW$-enriched categories and modules between them does have a completion
property with respect to freely adjoining lax colimits (collages); see \cite{85, CKW}.
What we have in mind is an objective version of the notion of {\em homological functor}
used by Andr\'e Joyal in \cite{homsymm}.

Here we are interested in the universal property of the construction of $\CW\text{-}\mathrm{Mod}$ from $\CW\text{-}\mathrm{Cat}$. As we pointed out in \cite{14}, cofibrations appear naturally in
this construction, so it is not surprising that they play a role here. In the first section, we review
cofibrations in bicategories and related concepts. In the second section we interpret all that for
enriched categories. We will restrict to the case where
the base $\CW$ has but one object and indeed reduces to a monoidal category 
$\CV$ so that we can refer to Kelly's book \cite{KellyBook} for background. 
 The usual inclusion $\CV\text{-}\mathrm{Cat}\to\CV\text{-}\mathrm{Mod}$ taking each
$\CV$-functor to a left adjoint $\CV$-module has certain properties which make it
a {\em homodular pseudofunctor} in a sense defined in the third section; indeed, it is the universal
such with domain $\CV\text{-}\mathrm{Cat}$.

In Section 4 we obtain the just mentioned universality as a special case of universality for a slightly modified 
version of pro-arrow equipment $\CK\to \CM$ in the sense of Richard Wood \cite{RJWEquipII}; 
we use the term ``module equipment''. For example, the inclusion $\CM_*\to \CM$
of the subbicategory of left adjoint morphisms into a finitary cosmos $\CM$ (in the sense
of \cite{46}) is universal homodular with domain $\CM_*$.

We briefly examine the behaviour of the monoidal
structures of coproduct $+$ and tensor $\otimes$ under the universality.  

The cocartesian monoidal bicategory $\CV\text{-}\mathrm{Mod}$ is actually {\em trace monoidal} in the obvious bicategorical version of traced monoidal in \cite{51}. We describe the $\mathrm{Int}$ construction of an autonomous monoidal bicategory $\mathrm{i}\CM$
from bicategories $\CM$ such as $\CV\text{-}\mathrm{Mod}$. 
The story is much the same as the sets and relations case described near the end of \cite{51}. 

\section*{\small{Acknowledgement}} Nathanael Arkor provided very helpful feedback, much like
a referee's report, on the first arXiv version of this paper by directing me to the references \cite{RoseWood_a, RoseWood_b, BKPVa, BKPVb, AT} and by suggesting further examples. The generalised
version of the main theorem that appears here was developed because of 
Nathanael's prodding me to include more of those examples.     

\tableofcontents

\section{Adjoints, coslices, bipushouts and cofibrations}

A morphism $C\xra{u} A$ in any bicategory $\CM$ is said to have right adjoint $A\xra{u^*}C$
when there exist 2-morphisms $\varepsilon : u\circ u^*\Ra 1_A$ and $\eta : 1_C\Ra u^*\circ u$,
called {\em counit} and {\em unit} respectively, such that the composites $u\xra{u\eta}uu^*u\xra{\varepsilon u}u$ and $u^*\xra{\eta u^*}u^*uu^*\xra{u^* \varepsilon}u^*$ are identities.
The notation is $u\dashv u^*$.
We also say $u$ is left adjoint to $u^*$. A left adjoint for $u$ may be denoted by $u^!$. 

\begin{lemma}\label{adjff}
Suppose $u^! \dashv u \dashv u^*$ in $\CM$. The unit for $u \dashv u^*$ is invertible if and only
if the counit for $u^! \dashv u$ is invertible.   
\end{lemma} 
\begin{proof}
If $\CM = \mathrm{Cat}$, each condition is equivalent to full faithfulness of the functor $u$.
The general result follows by a representability argument.   
\end{proof} 

Write $\CM_*$ for the subbicategory of $\CM$ with morphisms restricted to {\em maps}; that is, those morphisms $f$ with right adjoints $f^*$. Write $\mathcal{I} : \CM_*\to \CM$ for the inclusion pseudofunctor.  

For functors $A\xla{f} C\xra{g} G$, we will use the notation $f\small{\downarrow} g$ for the
comma category (a construction from \cite{LawLaJolla} with the present notation adopted in \cite{CWM}); however, we prefer the term {\em slice of $f$ over $g$}.
The objects are triples $(a, fa\xra{\gamma} gb, b)$ where $a\in A$
and $b\in B$ are objects and $\gamma$ is a morphism in $C$, the morphisms
$(a, fa\xra{\gamma} gb, b)\xra{(\alpha, \beta)} (a', fa'\xra{\gamma'} gb', b')$ consist of
morphisms $a\xra{\alpha}a'$ in $A$, $b\xra{\beta}b'$ in $B$ such that $\gamma' \circ f\alpha = g\beta \circ \gamma$. There are projection functors $f\small{\downarrow} g\xra{s} A$
and $f\small{\downarrow} g\xra{t} B$ with an obvious natural transformation $f\circ s\xra{\lambda} g\circ t$. Restricting the objects $(a, fa\xra{\gamma} gb, b)$ of $f\small{\downarrow} g$ to those
with $\gamma$ invertible, we obtain the pseudopullback (known also as the ``iso-comma category'').   

Slices can be defined in bicategories and the results are much as for the 2-category case considered in \cite{8}. In any bicategory $\CK$, the {\em slice (or comma object)} of a cospan $A\xra{f}C\xla{g}B$
is a diagram  
\begin{eqnarray}\label{slice}
\begin{aligned}
\xymatrix{
f\small{\downarrow} g \ar[d]_{s}^(0.5){\phantom{aaaaa}}="1" \ar[rr]^{t}  && B \ar[d]^{g}_(0.5){\phantom{aaaaa}}="2" \ar@{=>}"1";"2"^-{\ \lambda \ = \ \lambda_{f,g}}
\\
A \ar[rr]_-{f} && C 
}
\end{aligned}
\end{eqnarray}
with the property that pasting the diagram defines an equivalence of categories
\begin{eqnarray*}
\CK(X,f\small{\downarrow} g) \simeq \CK(X,f)\small{\downarrow} \CK(X,g) \ .
\end{eqnarray*}
If in this definition we had required $\lambda_{f,g}$ to be invertible and replaced
the comma category $\CK(X,f)\small{\downarrow} \CK(X,g)$ by the pseudopullback, 
we would have the definition of {\em bipullback} of the cospan.

Here is Proposition 5 of \cite{8}.
\begin{proposition}\label{sliceexact}
In the slice \eqref{slice}, if $f\dashv u$ then $t$ has a right adjoint $v$ with invertible
counit. Moreover, the mate $\bar{\lambda} : s\circ v\to u\circ g$ of $\lambda$ is invertible.
\begin{eqnarray}\label{exactslice}
\begin{aligned}
\xymatrix{
f\small{\downarrow} g \ar[d]_{s}^(0.5){\phantom{aaaaa}}="1"  && B \ar[ll]_{v} \ar[d]^{g}_(0.5){\phantom{aaaaa}}="2" \ar@{=>}"1";"2"^-{\bar{\lambda}}_-{\cong}
\\
A  && C \ar[ll]^-{u}
}
\end{aligned}
\end{eqnarray}
\end{proposition}

A {\em coslice} of a span $A\xla{u}C\xra{v}B$ is a diagram
\begin{eqnarray}\label{coslice}
\begin{aligned}
\xymatrix{
C \ar[d]_{v}^(0.5){\phantom{aaaaa}}="1" \ar[rr]^{u}  && A \ar[d]^{i}_(0.5){\phantom{aaaaa}}="2" \ar@{<=}"1";"2"^-{\gamma \ = \ \gamma_{u,v}}
\\
B \ar[rr]_-{j} && u\small{\uparrow} v 
}
\end{aligned}
\end{eqnarray}
which is a slice in $\CK^{\mathrm{op}}$.
Similarly, {\em bipushout} is bipullback in $\CK^{\mathrm{op}}$. 

 \begin{equation}\label{horizstackco}
 \begin{aligned}
\xymatrix{
C \ar[d]_{v}^(0.5){\phantom{aaaa}}="1" 
\ar[rr]^{u}  
&& A \ar[rr]^{n} 
     \ar[d]^{i}_(0.5){\phantom{aaaa}}="2" \ar@{<=}"1";"2"^-{\gamma} 
     \ar@{}[d]^{\phantom{aaaa}}="3"
&& E \ar[d]^{k}_(0.5){\phantom{aaaa}}="4" \ar@{<=}"3";"4"^-{\phi}_-{\cong}
\\
B \ar[rr]_-{j} && D \ar[rr]_-{m} && F 
}
 \end{aligned}
\end{equation}

Here is a dual of a fact appearing, for example, as Proposition 1.7 of \cite{46}.
\begin{proposition}\label{horizstackprop}
In the diagram \eqref{horizstackco}, suppose the left-hand square is a coslice. The
right-hand square is a bipushout if and only if the pasted diagram is a coslice of the span $E\xla{\ n \ \circ \ u}C\xra{\phantom{aa}v\phantom{aa}}B$.
\end{proposition}

As a consequence of a dual of Proposition~\ref{bipboffib} below, we have:
\begin{proposition}\label{anotherprop}
Take $B=C$ and $v=1_C$ in diagram \eqref{horizstackco} and suppose the left-hand square 
is a coslice and the right-hand square is a bipushout. If $n$ has a left adjoint $r$ then $m$ has a left adjoint $\ell$ satisfying $\ell\circ k\cong i\circ r$. If the counit of $r\dashv n$ is invertible then so
is the counit of $\ell \dashv m$.
\end{proposition}

The coslice $u \small{\uparrow} C$ of the span $A\xla{u}C\xra{1_C}C$ is also called the
{\em lax colimit} or {\em collage} of the diagram consisting of the single morphism $C\xra{u}A$. In this case, we can also write the coslice square as a triangle \eqref{collf}.
 \begin{equation}\label{collf}
 \begin{aligned}
\xymatrix{
C \ar[rd]_{j_u}^(0.5){\phantom{a}}="1" \ar[rr]^{u}  && A \ar[ld]^{i_u}_(0.5){\phantom{a}}="2" \ar@{<=}"1";"2"^-{\gamma}
\\
& u \small{\uparrow} C 
}
\end{aligned}
\end{equation}

As bicategorical and dual versions of Corollaries 6 and 7 of \cite{8}, we have:

\begin{proposition}\label{collprop}
For any collage \eqref{collf}, there is a right adjoint $i_u^*$ for $i_u$ with invertible unit 
$1_B\cong i_u^*\circ i_u$ and 
invertible mate $u \cong i_u^*\circ j_u$ of $\gamma$.
Moreover, $u$ has a right adjoint if and only if $j_u$ does; in this case, the unit of the
latter adjunction is invertible.
\end{proposition}

Suppose the bicategory $\CK$ admits slices. Each morphism $E\xra{p}A$ in $\CK$ induces a morphism
$E\small{\downarrow} E\xra{\bar{p}} A\small{\downarrow} p$ and isomorphisms $ps\cong s\bar{p}$
and $t \bar{p} \cong t$ which paste onto the 2-cell $\lambda_{A,p}$ to yield $p\lambda_{E,E}$
(where we often write an object in place of its identity morphism). 
The morphism $E\xra{p}A$ will be called a {\em fibration} (although it should properly be called a ``pseudofibration'') when $\bar{p}$ has a right adjoint with invertible counit in the bicategory $\CK$.
(In case $\CK$ is a 2-category it makes sense to ask that the counit should be an identity; we
then have a {\em strict fibration}.)
 \begin{equation}\label{horizstackex}
 \begin{aligned}
\xymatrix{
E\small{\downarrow} E \ar[d]_{s}^(0.5){\phantom{aaaa}}="1" 
\ar[rr]^{\bar{p}}  
&& A\small{\downarrow} E \ar[rr]^{t} 
     \ar[d]^{s}_(0.5){\phantom{aaaa}}="2" \ar@{=>}"1";"2"_-{\cong} 
     \ar@{}[d]^{\phantom{aaaa}}="3"
&& E \ar[d]^{p}_(0.5){\phantom{aaaa}}="4" \ar@{=>}"3";"4"^-{\lambda_{A,p}}
\\
E \ar[rr]_-{p} && A \ar[rr]_-{1_A} && A 
}
 \end{aligned}
\end{equation}

If the bicategory $\CK$ does not admit the appropriate slices then a morphism in $\CK$ is a {\em fibration}
when it is taken to a fibration by the bicategorical Yoneda embedding 
$\yo : \CK \to \mathrm{Hom}(\CK^{\mathrm{op}},\mathrm{Cat})$.  As in Item (1.3) of \cite{14},
we write $\mathrm{Hom}(\CA,\CX)$ for the bicategory of homomorphisms of bicategories (also called {\em pseudofunctors}) from $\CA$ to $\CX$, strong transformations (also called {\em pseudonatural transformations}), and modifications.  

Fibrations have the property that
pullbacks along them are automatically pseudopullbacks \cite{42} and hence bipullbacks
in $\mathrm{Cat}$.

\begin{proposition}\label{bipboffib}
If the diagram \eqref{bipb} is a bipullback and $p$ is a fibration in $\CK$ then $q$ is a fibration.
Moreover, if $f$ has a right adjoint $u$ then $r$ has a right adjoint $t$ with the mate $q\circ t\Ra u\circ p$ of $\theta$ invertible.
\begin{eqnarray}\label{bipb}
\begin{aligned}
\xymatrix{
F \ar[d]_{q}^(0.5){\phantom{aaaaa}}="1" \ar[rr]^{r}  && E \ar[d]^{p}_(0.5){\phantom{aaa}}="2" \ar@{=>}"1";"2"^{\theta}_-{\cong}
\\
C \ar[rr]_-{f} && A 
}
\end{aligned}
\end{eqnarray}
\end{proposition}

An {\em opfibration} is a fibration in $\CK^{\mathrm{co}}$. 
A {\em cofibration} is a fibration in $\CK^{\mathrm{op}}$.
A {\em coopfibration} is a fibration in $\CK^{\mathrm{coop}}$.

In diagram \eqref{slice}, the morphism $s$ is a fibration and $t$ is an opfibration.
In diagram \eqref{coslice}, the morphism $i$ is a cofibration and $j$ is a coopfibration.

Of course Proposition~\ref{bipboffib} has three duals.
\begin{eqnarray}\label{cosliceCandu}
\begin{aligned}
\xymatrix{
C \ar[d]_{u}^(0.5){\phantom{aaaaa}}="1" \ar[rr]^{1_C}  && C \ar[d]^{i_u}_(0.5){\phantom{aaa}}="2" \ar@{<=}"1";"2"^{\gamma}
\\
A \ar[rr]_-{j_u} && C\small{\uparrow}u 
}
\end{aligned}
\end{eqnarray}
In the coslice \eqref{cosliceCandu}, by a dual of Proposition~\ref{collprop}, $j_u$ has a left adjoint $j_{u !}$ with counit $j_{u !}\circ j_u \Ra 1_A$ invertible which, incidentally means that, if $j_u$ also had a right adjoint, the unit would be invertible.  Moreover, if $u$ has a right adjoint $u^*$
then $i_u$ has a right adjoint $i_u^*$ with invertible unit $1_C\Ra i_u^*\circ i_u$ invertible
and the mate $ j_{u !}\circ i_u\Ra u$ of $\gamma$ invertible.

If $\CK$ admits coslices, a morphism $C\xra{u}A$ is a cofibration if and only if the canonical
morphism $C\small{\uparrow}u\xra{\bar{u}} A\small{\uparrow}A$ induced by $\bar{u}$ has
a left adjoint $\ell$ with invertible counit $\ell \circ \bar{u}\cong 1_{C\small{\uparrow}u}$. 
\begin{eqnarray}\label{cofibadjunction} 
\xymatrix @R-3mm {
A\small{\uparrow}A \ar@<1.5ex>[rr]^{\ell} \ar@{}[rr]|-{\perp} && C\small{\uparrow}u \ar@<1.5ex>[ll]^{\bar{u}} \ar@<1.5ex>[rr]^{\ell\circ \bar{u}} \ar@{}[rr]|-{\Downarrow\cong} \ar@<-1.5ex>[rr]_{1_{C\small{\uparrow}u}} &&  C\small{\uparrow}u
}\end{eqnarray}
 
 \begin{lemma}\label{ffofcofibsinM*} For any bicategory $\CM$, suppose coslices of the form \eqref{cosliceCandu} exist in $\CM_*$. If $C\xra{u} A$ is a cofibration in $\CM_*$ then the unit of $u\dashv u^*$ is invertible.    
\end{lemma}
\begin{proof} We are in the situation of \eqref{cofibadjunction}: that is,
$C\small{\uparrow}u\xra{\bar{u}} A\small{\uparrow}A$ has a left adjoint with invertible counit.
However, $\bar{u}$ is in $\CM_*$ and so has a right adjoint for which, by Lemma~\ref{adjff},
the unit is invertible. From the remarks after diagram \eqref{cosliceCandu}, both 
$C\xra{i_u} C\small{\uparrow}u$ and $A\xra{i_A} A\small{\uparrow}A$ have right adjoints
with invertible units. Part of the definition of $\bar{u}$ is that we have an 
isomorphism $i_A\circ u \cong \bar{u}\circ i_u$ which implies that the unit for $u\dashv u^*$ is also invertible.   
\end{proof}

\section{The bicategory of modules}

Here we remind the reader of some results essentially from \cite{14, 22, 85, 46}.

So that we are in the familiar setting of Kelly's book \cite{KellyBook}, let $\CV$ denote a symmetric monoidal closed category which is compatibly locally presentable. We write
$\CV\text{-Cat}$ for the 2-category of (small) $\CV$-categories; the 1-morphisms are $\CV$-functors and the 2-morphisms are $\CV$-natural transformations.

We write $\CV\text{-Mod}$ for the bicategory of (small) $\CV$-categories; 
the hom-category $\CV\text{-Mod}(B,A)$ is the underlying category of the 
$\CV$-functor $\CV$-category $[A^{\mathrm{op}}\ot B,\CV]$ whose objects
we call ($\CV$-){\em modules from $B$ to $A$}. Composition of modules
$C\xra{n} B\xra{m} A$ is defined by the coend
\begin{eqnarray}
(m\circ n)(a,c) = \int^b n(b,c)\ot m(a,b) \ .
\end{eqnarray}
 
 There is an inclusion pseudofunctor 
 \begin{eqnarray}\label{lowstar}
(-)_* : \CV\text{-Cat}\to \CV\text{-Mod}
\end{eqnarray}
 which is the identity on objects and locally fully faithful on homs; it takes the $\CV$-functor $B\xra{f} A$ to the module $B\xra{f_*} A$ 
defined by 
\begin{eqnarray}
f_*(a,b) = A(a,fb) \ .
\end{eqnarray}
There is an adjunction $f_*\dashv f^*$ in the bicategory $\CV\text{-Mod}$ where
\begin{eqnarray}\label{upstar}
f^*(b,a) = A(fb,a) \ .
\end{eqnarray}
The inclusion $(-)_* : \CV\text{-Cat}\to \CV\text{-Mod}$, apart from a size problem,
has a right pseudo-adjoint $B\mapsto \mathcal{P}B = [B^{\mathrm{op}},\CV]$ which
suffices, in the usual way, to prove:
\begin{proposition}\label{lower*cocts}
The pseudofunctor \eqref{lowstar} preserves bicategorical colimits.
\end{proposition}
In particular, $(-)_*$ preserves coslices so a dual of Proposition~\ref{sliceexact} yields: 
\begin{corollary}\label{cosliceexact}
If the diagram \eqref{coslice} is a coslice in $\CV\text{-}\mathrm{Cat}$ then the mate
$\bar{\gamma}$ of the image of the 2-morphism $\gamma$ under the pseudofunctor \eqref{lowstar} is invertible.
\begin{eqnarray*}
\begin{aligned}
\xymatrix{
B \ar[d]_{v^*}^(0.5){\phantom{aaaaa}}="1" \ar[rr]^{j_*}  && u\small{\uparrow} v \ar[d]^{i^*}_(0.5){\phantom{aaaaa}}="2" \ar@{=>}"1";"2"^-{\bar{\gamma}}_-{\cong}
\\
C \ar[rr]_-{u_*} && A 
}
\end{aligned}
\end{eqnarray*}
\end{corollary} 

For a module $B\xra{m}A$, the coslice $m\small{\uparrow}B$ in $\CV\text{-Mod}$ is the $\CV$-category constructed in \cite{14} as follows. The set of objects is the disjoint union of the set of objects of $A$ and the set of objects of $B$. The hom $\CV$-objects of $m\small{\uparrow}B$ are defined by
\begin{eqnarray}\label{cosliceconstruction}
\begin{aligned}
(m\small{\uparrow}B)(x,y) = \left\{
\begin{array}{rl}
A(x,y) & \text{if } x, y \in A \\
B(x,y) & \text{if } x, y \in B \\
m(x,y) & \text{if } x \in A \text{ and } y\in B \\
 0 \phantom{aaa} & \text{otherwise } .
\end{array} \right.
\end{aligned}
\end{eqnarray}
Composition is obtained from the compositions in $A$ and $B$ and their actions on $m$.
We have fully faithful $\CV$-functors $A\xra{i_m}m\small{\uparrow}B$ and
$B\xra{j_m}m\small{\uparrow}B$ given by inclusion and a diagram
 \begin{equation}\label{collm}
 \begin{aligned}
\xymatrix{
B \ar[rd]_{j_{m *}}^(0.5){\phantom{a}}="1" \ar[rr]^{m}  && A \ar[ld]^{i_{m * }}_(0.5){\phantom{a}}="2" \ar@{<=}"1";"2"^-{\gamma}
\\
& m \small{\uparrow} B 
}
\end{aligned}
\end{equation}
where $\gamma$ is induced by the right action of $A$ on $m$.
If $m$ is of the form $f_*$ for a $\CV$-functor $B\xra{f}A$ then the right adjoint
$i_m^*$ of $i_{m*}$ is also in the image of $(-)_* : \CV\text{-Cat}\to \CV\text{-Mod}$ as expected
since $(-)_*$ preserves coslices.  

A full sub-$\CV$-category $A$ of a $\CV$-category $X$ is called a {\em sieve} (``crible'' in French) when, for all $x\in X$, if $x\notin A$ then $X(x,a)=0$ for all $a\in A$. 
From Gray \cite{GrayLaJolla} and Section 6 of \cite{14}, we have the following considerations.
The inclusions $A \hookrightarrow X$ of sieves are the strict cofibrations in $\CV\text{-}\mathrm{Cat}$. 
With such a sieve, let $B$ be the full sub-$\CV$-category of $X$ consisting of 
those objects $x\notin A$. 
Then the inclusions $A\xra{i}X\xla{j}B$ provide a collage
for the module $i^*\circ j_* = X(i,j) : B\to A$ in $\CV\text{-}\mathrm{Mod}$.
An {\em opsieve} is a sieve in $\CV\text{-}\mathrm{Cat}^{\mathrm{co}}$; in the last sentence, $j$ is such. 

Suppose we have a span $Y\xla{f}A\xra{i}X$ in $\CV\text{-}\mathrm{Cat}$ with $i$
the inclusion of a sieve. There is a commutative square 
\begin{eqnarray}\label{posieve}
\begin{aligned}
\xymatrix{
A \ar[d]_{f}^(0.5){\phantom{aaaaa}}="1" \ar[rr]^{i}  && X \ar[d]^{h}_(0.5){\phantom{aaaaa}}="2" \ar@{=}"1";"2"_-{ }
\\
Y \ar[rr]_-{k} && Q 
}
\end{aligned}
\end{eqnarray}  
in $\CV\text{-}\mathrm{Cat}$ defined as follows. The objects of $Q$ are those in the
disjoint union of the objects of $X$ not in $A$ and the objects of $Y$. 
The hom $\CV$-objects of $Q$ are defined by
\begin{eqnarray}\label{cosliceconstruction}
\begin{aligned}
Q(w,z) = \left\{
\begin{array}{rl}
Y(w,z) \phantom{aaaaaa} & \text{if } w, z \in Y \\
\int^{a\in A}X(a,z)\ot Y(w,fa) & \text{if } z\in X\setminus A \ \text{ and } w\in Y \\
 0 \phantom{aaaaaaaa} & \text{otherwise } .
\end{array} \right.
\end{aligned}
\end{eqnarray}
The $\CV$-functor $k$ is the inclusion while $hx$ is $x$ when $x\notin A$ and is $fx$
when $x\in A$. This construction has the following well-known and easily checked properties.

\begin{proposition}\label{po-sieve-cosieve}
The commutative square \eqref{posieve}, where $i$ is the inclusion of a sieve (respectively, fully faithful), is a pushout 
(and hence by \cite{42} a pseudopushout) in $\CV\text{-}\mathrm{Cat}$. Moreover, $k$ is also the inclusion of a sieve (respectively, fully faithful) and the canonical 
morphism $f_*\circ i^*\to k^*\circ h_*$ in $\CV\text{-}\mathrm{Mod}$ is invertible.\end{proposition}

\begin{proposition} Every cofibration in $\CV\text{-}\mathrm{Cat}$
is a fully faithful $\CV$-functor.
\end{proposition}
\begin{proof} For each cofibration $C\xra{u}A$ we have $j_A\circ u \cong \bar{u}\circ j_u$ 
(see Diagram~\ref{horizstackex} and dualise) with $\bar{u}, j_u, j_A$ all fully faithful. So $u$ is. 
\end{proof}
   
\begin{proposition} Every cofibration $A\to X$ in $\CV\text{-}\mathrm{Cat}$
is equivalent under $A$ to the inclusion of a sieve 
and every coopfibration $C\to Y$ is equivalent under $C$ to the inclusion of an opsieve.
\end{proposition}
 
\section{Homodular pseudofunctors}
\begin{eqnarray}\label{bipocofs}
\begin{aligned}
\xymatrix{
A \ar[d]_{j}^(0.5){\phantom{aaaaa}}="1" \ar[rr]^{i}  && X \ar[d]^{h}_(0.5){\phantom{aaaaa}}="2" \ar@{=>}"1";"2"_-{\cong}
\\
Y \ar[rr]_-{k} && Q 
}
\end{aligned}
\end{eqnarray}  

\begin{definition}\label{homoddef}  A pseudofunctor $T : \CK\to \CX$ between bicategories is {\em homodular} when
\begin{itemize}
\item[H1.] for all cofibrations $i$ in $\CK$, the morphism $Ti$ has a right adjoint $Ti^*$;
\item[H2.] for all bipushout squares \eqref{bipocofs} in $\CK$ with $i$ a cofibration and $j$ a coopfibration, the mate $Tj \circ Ti^* \Ra Tk^*\circ Th$, of the image under $T$ of the isomorphism in the square, is invertible.   
\end{itemize}
\end{definition}

Before providing examples, we give some fairly obvious constructions of new homodular pseudofunctors from old.

A pseudofunctor $F : \CX \to \CH$ is {\em fully faithful} when its effect $F :\CX(X, Y)\to \CH(FX,FY)$ on hom
categories is an equivalence of categories for all $X, Y\in \CX$.

\begin{proposition}\label{ffcomphomod} Suppose $F : \CX \to \CH$ is a fully faithful pseudofunctor. The pseudofunctor $T : \CK\to \CX$ is homodular if and only if the composite $F\circ T : \CK\to \CH$ is homodular.
\end{proposition} 

\begin{proposition}\label{homodprecomp} Suppose $\CH$ has bipushouts and coslices preserved by $S : \CH\to \CK$. If $T : \CK\to \CX$ is homodular then so is $T\circ S : \CH\to \CX$.
\end{proposition}

\begin{proposition}\label{homintohomod} If $T : \CK\to \CX$ is homodular then, for any bicategory $
\CA$, so is $$T\circ - : \mathrm{Hom}(\CA,\CK)\to \mathrm{Hom}_{\mathrm{lnt}}(\CA,\CX) \ ,$$
where the $\mathrm{lnt}$ subscript indicates that the morphisms are lax natural transformations
between pseudofunctors. 
\end{proposition}

We are mainly interested in homodular pseudofunctors with cocomplete bicategories for domain.
So the following is worth noting.

\begin{proposition} Suppose $T: \CK\to \CX$ is a homodular pseudofunctor where $\CK$ admits coslices. Then, for all $B\xra{f} A$ in $\CK$, the morphism $Tf$ has a right adjoint $Tf^*$.   
\end{proposition}
\begin{proof} By a dual of Proposition~\ref{collprop}, in the coslice diagram
\begin{eqnarray*}
\begin{aligned}
\xymatrix{
B \ar[d]_{f}^(0.5){\phantom{aaaaa}}="1" \ar[rr]^{1_B}  && B \ar[d]^{i}_(0.5){\phantom{aaaaa}}="2" \ar@{<=}"1";"2"^-{\gamma}
\\
A \ar[rr]_-{j} && B\small{\uparrow} f  \ ,
}
\end{aligned}
\end{eqnarray*}
the morphism $i$ is a cofibration while $j$ has a left adjoint $j'$ with $j'\circ j\cong 1_A$ and $j'\circ i\cong f$.
So $Tf \cong Tj'\circ Ti$ has right adjoint $Ti^*\circ Tj$ using condition H1.
\end{proof}

Now for some examples.

\begin{example}\label{exparadigmatic} The paradigmatic example is
$(-)_* : \mathrm{Cat} \to \mathrm{Mod}$ 
where $\mathrm{Mod} : = \mathrm{Set}\text{-}\mathrm{Mod}$.
More generally, we have $(-)_* : \CV\text{-Cat}\to \CV\text{-Mod}$ of \eqref{lowstar}. 
\end{example}

\begin{example}\label{functorsintoX} Suppose $X$ is a complete category. Then 
$[-,X] : \mathrm{Cat}^{\mathrm{op}} \to \mathrm{CAT}$ is homodular.
Since pointwise right Kan extensions of morphisms into $X$ exist, we have condition H1.   
Pointwise right (Kan) extensions along fibrations satisfy H2 by a dual of Proposition 23 of \cite{8}.
\end{example}

\begin{example}\label{Catexample} The pseudofunctor 
$(-)^* : \mathrm{Cat}^{\mathrm{op}} \to \mathrm{Mod}^{\mathrm{co}}$ is homodular. 
One way to see this is to compose with the fully faithful pseudofunctor 
$\widetilde{(-)} : \mathrm{Mod}^{\mathrm{co}} \to \mathrm{CAT}$ taking a small category $A$
to the functor category $[A,\mathrm{Set}]$ and the module $B\xra{m}A$ to the functor
$[B,\mathrm{Set}]\xra{\widetilde{m}} [A,\mathrm{Set}]$ defined by $\widetilde{m}v = \int_b[m(-,b),vb]$.
The composite is $[-,\mathrm{Set}] : \mathrm{Cat}^{\mathrm{op}} \to \mathrm{CAT}$ and so
is homodular by Example~\ref{functorsintoX}. Then Proposition~\ref{ffcomphomod} gives the result.     
\end{example}

\begin{example}\label{RoseWoodex} Let $\mathrm{TOP}$ denote the bicategory of topoi and geometric morphisms. Let $\mathrm{TLEX}$ be the bicategory of topoi and finite-limit-preserving (= left-exact) functors. The pseudofunctor 
$(-)_* : \mathrm{TOP} \to \mathrm{TLEX}^{\mathrm{co}}$, taking each geometric 
morphism to its left-exact left adjoint, is homodular. 
This follows from Section 2 of Rosebrugh-Wood \cite{RoseWood} which identifies cofibrations in $\mathrm{TOP}^{\mathrm{op}}$ as fibrations in $\mathrm{CAT}$ with
an extra property. Similarly, we can see from \cite{RoseWood_a} that $(-)_* : \mathrm{ABEL} \to \mathrm{ALEX}^{\mathrm{co}}$ provides an example for abelian categories in place of topoi.      
\end{example}

\section{Universal homodularity and equipment}

\begin{definition}\label{universalhomod}
A homodular pseudofunctor $S : \CK\to \CM$ is {\em universal} when, for each 
homodular pseudofunctor $T : \CK \to \CX$, there is a pseudofunctor $\bar{T} : \CM\to \CX$ which
is unique up to equivalence with the property $\bar{T}\circ S \simeq T$.
\end{definition}

A pseudofunctor $E : \CK\to \CM$ is {\em essentially surjective on objects} when, for each object $X\in \CM$,
 there exists an object $K\in \CK$ with $EK\simeq X$.
 
 \begin{proposition} Universal homodular pseudofunctors are essentially surjective on objects. 
 \end{proposition}
\begin{proof}
Suppose $S: \CK\to \CD$ is a universal homodular pseudofunctor.
Factor it as $S\simeq F\circ E$ where $E : \CK\to \CM$ is the identity on objects and $F : \CM\to \CD$ is fully faithful.
By Example~\ref{ffcomphomod}, $E$ is homodular. Since $S$ is universal, there exists a pseudofunctor 
$R : \CD\to \CM$ with $R\circ S\simeq E$. Then $F\circ R\circ S \simeq 1_{\CD}\circ S$ and the uniqueness
clause for universality imply $F\circ R\simeq 1_{\CD}$. By uniqueness of the factorization, $R\circ F\circ E\simeq E$ and $F\circ R\circ F\simeq F$ imply $R\circ F \simeq 1_{\CM}$. So $F$ is a biequivalence with the result that $S\simeq F\circ E$ is essentially bijective on objects. 
\end{proof}

Recall from Section 6 of \cite{46} that an {\em finitary cosmos} is a bicategory $\CM$ possessing
  \begin{itemize}
 \item[(a)]  finite bicategorical coproducts;
 \item[(b)] a Kleisli object for every monad;
 \item[(c)] finite colimits in the hom categories preserved by composition on both sides.
  \end{itemize}
 As mentioned atop page 266 of \cite{46}, the construction proving Proposition 1 of \cite{85} shows that any finitary cosmos admits finite collages. 
   
 The following definition connects with the notion of {\em proarrow equipment} due to Wood \cite{RJWEquipI, RJWEquipII}. Indeed, as pointed out by Nathanael Arkor, the axioms of \cite{RoseWood_b} imply that $\CM$ is a finitary cosmos.
 
    \begin{definition}\label{defnequip} {\em Module equipment} for a bicategory $\CK$ consists of a
    finitary cosmos $\CM$ and a pseudofunctor $\mathcal{Q} : \CK\to \CM$
    with the following properties: 
  \begin{itemize}
   \item[(o)] if $i$ is a cofibration in $\CK$ then $\mathcal{Q}i$ has a right adjoint $\mathcal{Q}i^*$, 
  \item[(i)] $\mathcal{Q}$ is essentially surjective on objects,
  \item[(ii)] $\mathcal{Q}$ is locally fully faithful (meaning, fully faithful on homcategories),
  \item[(iii)] a morphism out of any finite collage in $\CM$ is in the essential image of 
$\mathcal{Q}$ if and only if its composites with all the injections in the lax cocone are in the essential image of $\mathcal{Q}$.
   \end{itemize}
    \end{definition}

We will use properties (i) and (ii) to abuse notation by dropping mention of $\mathcal{Q}$.
So when we say a morphism $B\xra{m}A$ of $\CM$ is in $\CK$ we mean that it can be replaced
up to equivalence by a morphism $\mathcal{Q}L\xra{\mathcal{Q}k}\mathcal{Q}K$ where $L\xra{k}K$ 
is in $\CK$. This is intended to help readers by simplifiying notation.

For example, applying the property (iii) assumed about collages to the identity morphism of a collage, we allow ourselves to say that the injections of the lax cocone are in $\CK$.

\begin{proposition}\label{uupCandCupu} If a bicategory $\CK$ admits module equipment $\mathcal{Q}$ then coslices exist in $\CK$ and are preserved by $\mathcal{Q}$.
\end{proposition}
\begin{proof} Take $A\xla{u}C\xra{v}B$ in $\CK$. Take the collage of the morphism $uv^*$ in $\CM$ to obtain $uv^*\small{\uparrow}B$ which is equivalent to $u\small{\uparrow}v$ using the 
calculus of mates under adjunction to verify the universal property. By Definition~\ref{defnequip} Item (iii), this coslice is in $\CK$.  
\end{proof}
 
\begin{corollary} Module equipments preserve cofibrations.
\end{corollary} 
\begin{proof} Notice that the definition of cofibration (see $\eqref{cofibadjunction}$) involves adjunctions between coslices of the form $C\small{\uparrow}v$.
\end{proof} 

  The proof of Proposition 6.8 of \cite{46} can be seen to actually prove Proposition~\ref{46Prop6.8}
  given Definition~\ref{defnequip} Item (iii). 
  
    \begin{proposition}\label{46Prop6.8} Suppose $X\xla{i}A\xra{j}Y$ is a span in $\CK$ with module equipment $\mathcal{Q} : \CK\to \CM$. Suppose in $\CM$ the unit of $i\dashv i^*$ is invertible. Then a bipushout \eqref{bipocofs} exists in $\CK$ and is preserved by the pseudofunctor $\mathcal{Q} : \CK\to \CM$ .
    Moreover, in $\CM$, $k\dashv k^*$ has invertible unit and the mate $j \circ i^* \Ra k^*\circ h$ of $k\circ j\cong h\circ i$ is invertible.
 \end{proposition}
     
  \begin{theorem}\label{univhomodforM*} 
The pseudofunctor $\mathcal{Q} : \CK\rightarrow \CM$ of module equipment for $\CK$ is universal homodular out of $\CK$.    
\end{theorem}
\begin{proof} 
Clearly H1 is satisfied since $\mathcal{Q}$ factors through $\CM_*$ and pseudofunctors preserve adjunctions. Proposition~\ref{46Prop6.8} gives H2.  

Now take a homodular pseudofunctor $T : \CK\to \CX$. For each morphism $B\xra{m}A$ in $\CM$, define $\bar{T}m = (Ti_m)^*\circ Tj_m$ in the notation of \eqref{collm}. This definition is forced since we have $m\cong i_m^*\circ j_{m}$ in $\CM$ and $\bar{T}$ is to be a pseudofunctor essentially restricting to $T$.  

We must see that $\bar{T}$ coherently preserves composition up to isomorphism. So consider composable morphisms $C\xra{n}B\xra{m}A$.
 
Put $z : = (C\xra{m\circ n}A)$ and $y : = (C\xra{j_m\circ n}m \small{\uparrow} B)$ and form their collages. 
     \begin{equation}\label{collmcomps}
 \begin{aligned}
\xymatrix{
C \ar[rd]_{j_z}^(0.5){\phantom{a}}="1" \ar[rr]^{z}  && A \ar[ld]^{i_z}_(0.5){\phantom{a}}="2" \ar@{<=}"1";"2"^-{\gamma_z}
\\
& z \small{\uparrow} C 
}
\quad
\xymatrix{
C \ar[rd]_{j_y}^(0.5){\phantom{a}}="1" \ar[rr]^{y}  && m \small{\uparrow} B \ar[ld]^{i_y}_(0.5){\phantom{a}}="2" \ar@{<=}"1";"2"^-{\gamma_y}
\\
& y \small{\uparrow} C 
}
\end{aligned}
\end{equation}
By the universal property of $n \small{\uparrow} C$, there exist $n \small{\uparrow} C\xra{h}y \small{\uparrow} C$ and $i_y\circ j_m \cong h\circ i_n$, $h\circ j_n\cong j_y$ which paste
to $\gamma_n$ to yield $\gamma_y$.
Again by the universal property of $n \small{\uparrow} C$, there exist $n \small{\uparrow} C\xra{k}z \small{\uparrow} C$ and $i_z\circ m \cong k\circ i_n$, $k\circ j_n\cong j_z$ which paste
to $\gamma_n$ to yield $\gamma_z$. By the universal property of $z \small{\uparrow} C$, there exist $z \small{\uparrow} C\xra{\ell}y \small{\uparrow} C$ and $i_y\circ i_m \cong \ell\circ i_z$, $\ell \circ j_z\cong j_y$ which paste
to $\gamma_z$ to yield $\gamma_y$. 
By the assumed property Definition~\ref{defnequip} Item (iv) of collages, each of $h$, $k$ and $\ell$ is in $\CK$.

By the universal property of $y \small{\uparrow} C$, there exist 
$y \small{\uparrow} C\xra{\bar{\ell}}z \small{\uparrow} C$ and 
$i_z\circ i_m^* \cong \bar{\ell}\circ i_y$, $\bar{\ell} \circ j_y\cong j_z$ which paste
to $\gamma_y$ to yield $\gamma_z$. By Proposition~\ref{horizstackprop}, the right two
squares in diagram \eqref{compMspres} are bipushouts. By Proposition~\ref{anotherprop},
using that $i_m^*$ has a left adjoint $i_m$ with invertible unit, we see that $\bar{\ell}\cong\ell^*$
and the unit of $\ell\dashv\ell^*$ is invertible. (To be more explicit, the restriction along $i_z$ of the unit is 
$i_z\cong i_zi^*_zi_m\cong \bar{\ell}i_yi_m\cong\bar{\ell}\ell i_z$ while the restriction along $j_z$ is 
$j_z\cong \bar{\ell}j_y\cong \bar{\ell}\ell j_z$; the restriction along $i_y$ of the counit is $\ell\bar{\ell}i_y
\cong \ell i_z i_m^*\cong i_yi_mi^*_m \xra{i_y\varepsilon}i_y$ while the restriction along $j_y$ is   
$\ell\bar{\ell}j_y\cong \ell j_z\cong j_y$.)

\begin{equation}\label{compMspres}
 \begin{aligned}
\xymatrix{
C \ar[d]_{C}^(0.5){\phantom{aaaaa}}="1" 
\ar[rr]^{n}  
&& B \ar[rr]^{j_{m}} 
     \ar[d]^{i_{n}}_(0.5){\phantom{aaaaa}}="2" \ar@{<=}"1";"2"_-{\gamma_n} 
     \ar@{}[d]^{\phantom{aaaaaaa}}="3"
&& m\small{\uparrow}B \ar[rr]^{i_m^*} \ar[d]^{i_y}_(0.5){\phantom{aaaaa}}="4" \ar@{<=}"3";"4"^-{\phantom{a}}_-{\cong}
\ar@{}[d]^{\phantom{aaaaa}}="5"
&& A \ar[d]^{i_{z}}_(0.5){\phantom{aaaaaaa}}="6" \ar@{<=}"5";"6"^-{\phantom{a}}_-{\cong}
\\
C \ar[rr]_-{j_{n}} && n\small{\uparrow}C \ar[rr]_-{h} && y\small{\uparrow}C \ar[rr]_{\bar{\ell}} && z\small{\uparrow}C 
}
 \end{aligned}
\end{equation}

We now have the isomorphisms
\begin{eqnarray*}
\bar{T}m\circ \bar{T}n & = &  (Ti_m)^*\circ Tj_m\circ (Ti_n)^*\circ Tj_n \\
& \cong &  (Ti_m)^*\circ (Ti_y)^*\circ Th \circ Tj_n \\
& \cong &  T(i_y\circ i_m)^*\circ T(h\circ j_n) \\
& \cong &  T(\ell \circ i_{z})^*\circ T(\ell\circ j_{z}) \\
& \cong &  T(i_z)^*\circ (T\ell)^*\circ T\ell \circ Tj_{z} \\
& \cong &  T(i_{z})^*\circ Tj_{z} \\
&  = &  \bar{T}(m\circ n) \ . 
\end{eqnarray*}

We must see that $\bar{T}(u) \cong Tu$ for any morphism $u\in \CK$ which will also imply
that $\bar{T}$ preserves identity morphisms up to isomorphism. We use Proposition~\ref{collprop} to obtain isomorphisms
\begin{eqnarray*}
\bar{T}u \ = \ (Ti_{u})^* \circ Tj_{u} \ \cong \ T(i_{u}^*) \circ Tj_u \ \cong \ T(i_{u}^* \circ j_u)  \
 \cong \ Tu  \ .
\end{eqnarray*}

Coherence for these isomorphisms can be verified.    
\end{proof}

\begin{example} If $\CM$ is a finitary cosmos, the inclusion pseudofunctor $\mathcal{I} : \CM_*\to \CM$ is universal homodular. This is the case where $\mathcal{Q}= \mathcal{I}$ in Theorem~\ref{univhomodforM*}. 
\end{example}

\begin{example}\label{universalityofmod}  
The universal homodular pseudofunctor with domain $\CV\text{-Cat}$
is the $(-)_* : \CV\text{-Cat}\to \CV\text{-Mod}$ of \eqref{lowstar}.     
\end{example}

Example~\ref{universalityofmod} together with Corollary~\ref{cosliceexact} yield:
 \begin{corollary} For any homodular pseudofunctor $T : \CV\text{-Cat} \to \CX$ and
 coslice \eqref{coslice} in  $\CV\text{-}\mathrm{Cat}$, the mate $Tu\circ Tv^*\Lra Ti^* \circ Tj$
 of $T\gamma$ is invertible.  
 \end{corollary}
 
 The terminology of the next result is that of \cite{60}.

 \begin{proposition}\label{monoidalinduced} Suppose $\CX$ is a monoidal bicategory and 
 $T : \CV\text{-Cat} \to \CX$ is a homodular pseudofunctor with extension 
 $\bar{T} : \CV\text{-Mod} \to \CX$. 
 If $T$ is strong monoidal then so is $\bar{T}$ such that the canonical pseudonatural equivalence
 $\bar{T}\circ (-)_*\simeq T$ is monoidal.  
 \end{proposition}
\begin{proof}(Sketch) The main point is that $i_m\ot 1_{A'} \cong i_{m\ot 1_{A'}}$ and $j_m\ot 1_{A'} \cong j_{m\ot 1_{A'}}$ for $B\xra{m}A$ in $\CV\text{-Mod}$ since $-\ot A'$ preserves coslices. So
\begin{eqnarray*}
\bar{T}(m \ \ot \ 1_{A'}) & = & (Ti_{m\ot 1_{A'}})^* \circ Tj_{m\ot 1_{A'}} \\
& \cong & T(i_{m}\ot 1_{A'})^* \circ T(j_{m}\ot 1_{A'}) \\
& \cong & (Ti_{m}^*\ot 1_{TA'}) \circ (Tj_{m}\ot 1_{TA'}) \\
& = & \bar{T}m \ \ot \ TA' \ .
\end{eqnarray*}
Symmetrically, $\bar{T}(1_A \ \ot \ m') \cong TA \ \ot \ \bar{T}m'$. Thus $\bar{T}(m \ \ot \ m') \cong \bar{T}m \ \ot \ \bar{T}m'$.  
\end{proof}

\begin{proposition}\label{coprodinduced} Suppose $T : \CV\text{-Cat} \to \CX$ is a homodular pseudofunctor with extension $\bar{T} : \CV\text{-Mod} \to \CX$. 
 If $T$ preserves coproducts then $\bar{T}$ preserves coproducts and preserves coproducts
 in the homcategories.  
 \end{proposition}
\begin{proof}(Sketch) Put $\CM = \CV\text{-Mod}$ for the time being. We treat the case of binary coproducts. We have the collage:
\begin{equation}\label{collmcop}
 \begin{aligned}
\xymatrix{
B \ar[rd]_{j_{0 *}}^(0.5){\phantom{a}}="1" \ar[rr]^{0}  && A \ar[ld]^{i_{0 * }}_(0.5){\phantom{a}}="2" \ar@{<=}"1";"2"^-{\gamma}
\\
& A+B 
}
\end{aligned}
\end{equation}
and the direct sum
\begin{eqnarray}\label{directsum} 
\xymatrix @R-3mm {
A \ar@<1.5ex>[rr]^{i_{0*}}\ar@{}[rr]|-{\perp}  &&A+B \ar@<-1.5ex>[rr]_{j_0^*} \ar@{}[rr]|-{\perp} \ar@<1.5ex>[ll]^{i_0^*}   &&  B \ar@<-1.5ex>[ll]_{j_{0 *}}  
}
\end{eqnarray}
in $\CM$: the units of the adjunctions are invertible, the counits induce an isomorphism
\begin{eqnarray*}
i_{0*}\circ i_i^* + j_{0*}\circ j_i^* \cong 1_{A+B}
\end{eqnarray*}
 while $i_0^*\circ j_{0*}\cong 0$ and $j_0^*\circ i_{0*}\cong 0$. Since $\bar{T}i_0 = Ti_0$,
$\bar{T}j_0 = Tj_0$ and $T$ preserves the coproduct, it follows that $\bar{T}$ preserves
the coproduct. 

The codiagonal $\nabla_A : A+A\to A$ is a $\CV$-functor and so $\nabla_{A *}$ has a right
adjoint $\nabla_{A}^* : A\to A+A$ in $\CM$. The top row composite of \eqref{aboutcoproducts} takes $A\xra{v}B$ to $$A+B\xra{v+v}A+B\xra{\nabla_B}B$$ and, by pseudonaturality of $\nabla$, has left
adjoint $\CM(\nabla_{A}^*,B)$. The bottom row composite takes $TA\xra{w}TB$ to 
$$T(A+A)\simeq TA+TA \xra{w+w}TB+TB\simeq T(B+B)\xra{T\nabla_B}TB$$ 
and has left adjoint $\CM(T\nabla_{A}^*,B)$. 
 \begin{equation}\label{aboutcoproducts}
 \begin{aligned}
\xymatrix{
\CM(A,B) \ar[d]_{\bar{T}}^(0.5){\phantom{aaaaaaaaa}}="1" 
\ar[r]^{\Delta \phantom{aaaaaa}}  
& \CM(A,B)\times \CM(A,B) \ar[r]^{\simeq} 
     \ar[d]^{\bar{T}\times \bar{T}}_(0.5){\phantom{aaaaaaaaaaaa}}="2" \ar@{=>}"1";"2"_-{\cong} 
     \ar@{}[d]^{\phantom{aaaaaaaaaaaa}}="3"
& \CM(A+A,B) \ar[d]^{\bar{T}}_(0.5){\phantom{aaaaaaaaaaaa}}="4" \ar@{=>}"3";"4"_-{\cong}
\\
\CX(TA,TB) \ar[r]_-{\Delta} & \CX(TA,TB)\times \CX(TA,TB) \ar[r]_-{\simeq} & \CX(T(A+A),TB) 
}
 \end{aligned}
\end{equation} 
The diagram made up of the left and right sides of \eqref{aboutcoproducts} and these left adjoints
clearly commutes up to a canonical isomorphism. It follows that the top and bottom $\Delta$s of the
left square have left adjoints which commute up to isomorphism with the vertical sides. Of course,
left adjoints to diagonal functors $\Delta$ are given by coproduct and the pseudocommutativity
means $\bar{T}$ preserves them. Incidentally this means $\bar{T}$ takes the direct sum \eqref{directsum} to a direct sum in $\CX$.    
\end{proof}
 
\section{On a question of Zeinab Galal}

After my talk in the Pacific Category Theory Seminar (8 May 2026), Zeinab Galal asked whether the following two results could be transported to a general homodular functor?

\begin{theorem}[James Worrell \cite{Wor}]\label{Thm1} If $F : \CV\text{-}\mathrm{Cat} \to \CV\text{-}\mathrm{Cat}$ preserves
full faithfulness of $\CV$-functors then the formula for extending $F$ to an $\tilde{F} : \CV\text{-}\mathrm{Cat} \to \CV\text{-}\mathrm{Cat}$, using the collage of a module, gives a lax double functor $(F,\tilde{F})$ on the double category $\CV\text{-}\mathbf{Mod}$ (where the two types of morphisms
are $\CV$-functors and $\CV$-modules). 
\end{theorem}

\begin{theorem}[Bilkova et al. \cite{BKPVb}]\label{Thm2} If $F : \CV\text{-}\mathrm{Cat} \to \CV\text{-}\mathrm{Cat}$ preserves exact squares then $(F,\tilde{F})$ is a pseudo-double functor. 
\end{theorem}

Theorems~\ref{ffthm} and \ref{exactthm} were inspired by that question. 
 
 \begin{definition}\label{homoddef}  If $T :\CK\to \CX$ is a pseudofunctor, we say a morphism 
$f\in \CK$ is {\em $T$-fully faithful} when $Tf$ has a right adjoint $Tf^{*}$ with invertible unit.
\end{definition}

\begin{theorem}\label{ffthm} Suppose $S : \CK \to \CM$ is module equipment and $T :\CK\to \CX_*$ is a pseudofunctor. If every $S$-fully faithful morphism of $\CK$ is also $T$-fully faithful then $\bar{T} : \CM\to \CX$ is a normal lax functor extending $T$. 
\end{theorem}
\begin{proof} In the proof of Theorem 4.7 of \cite{46} in the calculation displayed below diagram (4.20), we see that at the second step we only have a morphism since the Chevalley-Beck condition is not assumed anymore. At step six the full faithfulness condition comes in. In the final calculation of that proof we do still have isomorphisms; so $\bar{T}$ is normal. 
\end{proof}
Theorem~\ref{Thm1} is obtained by applying Theorem~\ref{ffthm} to the left-hand diagram of \eqref{extend}
with $T= S\circ F$.

\begin{eqnarray}\label{extend}
\begin{aligned}
\xymatrix{
\CK \ar[d]_{F}^(0.5){\phantom{aaaaa}}="1" \ar[rr]^{S}  && \CM \ar[d]^{\tilde{F}}_(0.5){\phantom{aaaaa}}="2" \ar@{=>}"1";"2"_-{\simeq}
\\
\CK \ar[rr]_-{S} && \CM 
} \
\quad \phantom{aa}
\
\xymatrix{
A \ar[d]_{j}^(0.5){\phantom{aaaaa}}="1" \ar[rr]^{i}  && X \ar[d]^{h}_(0.5){\phantom{aaaaa}}="2" \ar@{=>}"1";"2"_-{\gamma}
\\
Y \ar[rr]_-{k} && Q 
}
\end{aligned}
\end{eqnarray}  
\begin{definition}\label{homoddef}  If $T :\CK\to \CX$ is a pseudofunctor, we say the right-hand square 
of \eqref{extend} in $\CK$ is {\em $T$-exact} when $Ti$ and $Tk$ have right adjoints $Ti^{*}$ and $Tk^{*}$, and the mate $Tj \circ Ti^* \Ra Tk^*\circ Th$ of $T\gamma$ is invertible.
\end{definition}
\begin{theorem}\label{exactthm} Suppose $S : \CK \to \CM$ is universal homodular and $T :\CK\to \CX$ is a pseudofunctor landing in $\CX_*$. If every $S$-exact square in $\CK$ is also $T$-exact then $\bar{T} : \CM\to \CX$ is a pseudofunctor extending $T$ along $S$. 
\end{theorem}
\begin{proof} Under these conditions $T$ is homodular so universality of $S$ can be used.
Condition H1 for homodularity is clear since $T$ lands in $\CX_*$. Condition H2 says that
certain diagrams should be exact with respect to the pseudofunctor involved so its holding
for $S$ implies it does for $T$.
\end{proof}  
Theorem~\ref{Thm2} is obtained by applying Theorem~\ref{exactthm} to the left-hand diagram of \eqref{extend} with $T= S\circ F$.
  
 \section{The Int construction for modules}
 
 In Section 6 of \cite{51}, we constructed a symmetric autonomous monoidal 
 category $\mathrm{IntRel}$ from the category $\mathrm{Rel}$ of sets and relations with 
 disjoint union of sets for the tensor. A similar construction can be made 
 starting with a finitary cosmos $\CM$ where condition (c) is strengthened to:
 \begin{itemize}
  \item[(c$_{\omega}$)] countable colimits in the hom categories preserved by composition on both sides.
  \end{itemize} 
 This makes $\CM$ an iterative cosmos in the sense of \cite{46}. We will denote the resultant autonomous monoidal bicategory (in the sense of \cite{60}) by $\mathrm{i}\CM$.
 
 We begin by recalling that coproduct $A+B$ is direct sum in in $\CM$; it is both product and coproduct:
\begin{eqnarray*}
\CM(A+B,C)\simeq \CM(A,C)\times \CM(B,C) , \phantom{a} \CM(C,A+B)\simeq \CM(C,A)\times \CM(C,B) \ .
\end{eqnarray*}
The right adjoints $A\xla{\mathrm{in}_1^*} A+B \xra{\mathrm{in}_2^*}B$ of the injections $A\xra{\mathrm{in}_{1}} A+B \xla{\mathrm{in}_{2}}B$
are the projection morphisms. Given a module $R : B_1 + B_2 \to A_1 + A_2$, we obtain a matrix
\begin{eqnarray*}
\begin{bmatrix}
R_{11} & R_{12} \\
R_{21} & R_{22}
\end{bmatrix}
\end{eqnarray*}
of modules where $R_{ij} = (B_j\xra{\mathrm{in}_{i}} B_1 + B_2\xra{R}A_1 + A_2\xra{\mathrm{in}_i^*}A_i)$.

The free monad $R^{\odot}$ on an endomodule $R : U\to U$ is defined by the geometric 
series
\begin{eqnarray*}
R^{\odot} = \sum_{n=0}^{\infty}R^{\circ n} = U + R + R\circ R + R\circ R\circ R + \dots 
\end{eqnarray*}
which works because composition preserves local countable coproducts.
Each morphism $\rho : R\to T$ functorially induces a monad morphism $\rho^{\odot} : R^{\odot}\to T^{\odot}$.  

Here is the appropriate version of Lemma 6.1 of \cite{51}. 

\begin{lemma}\label{6.1of51}
For $R, P :U\to V$ and $S : V\to U$ in $\CM$, there are canonical isomorphisms:
\begin{itemize}
\item[(i)] $(R+S)^{\odot} \cong (R^{\odot}\circ S)^{\odot}\circ R^{\odot}$ for $U=V$;
\item[(ii)] $(R\circ S)^{\odot} \cong U+R\circ (S\circ R)^{\odot}\circ S$;
\item[(iii)] $R^{\odot}\cong U + R\circ R^{\odot}$ for U = V;
\item[(iv)] $R^{\odot}\circ R \cong R\circ R^{\odot}$ for U = V;
\item[(v)] $(R\circ S)^{\odot}\circ R \cong R\circ (S\circ R)^{\odot}$;
\item[(vi)] $(U\xra{0}U)^{\odot} \cong (U\xra{U}U)$.
\end{itemize}
\end{lemma}
 
 This brings us to the definition of the bicategory $\mathrm{i}\CM$. The objects are
 pairs $(X,U)$ of objects of $\CM$. Morphisms $R : (X,U)\to (Y,V)$ are morphisms $X+V\to Y+U$ in $\CM$ with matrices
 \begin{eqnarray*}
\begin{bmatrix}
A & B \\
C & D
\end{bmatrix}
\end{eqnarray*}
depicted by diagrams in $\CM$ as in \eqref{iMmorph}.
 \begin{eqnarray}\label{iMmorph}
\begin{aligned}
\xymatrix{
X \ar[rr]^-{A} \ar[d]_-{C} && Y  \\
U  && V \ar[u]_-{B} \ar[ll]^-{D}}
\end{aligned}
\end{eqnarray}
Composition of morphisms is as shown in diagram \eqref{iModmorphcomp}
 (where the composition symbol $\circ$ is omitted to save space).
\begin{eqnarray}\label{iModmorphcomp}
\begin{aligned}
\xymatrix{
X \ar[r]^-{A} \ar[dd]_-{C} & Y \ar[r]^-{E} \ar@/^/[dd]^-{G}  & Z  \\
\\
U & V  \ar[l]^-{D} \ar@/^/[uu]^-{B}  & W \ar[uu]^-{F}  \ar[l]^-{H}
}
\qquad  \xymatrix{ \\ = \phantom{a} }
 \xymatrix{
X \ar[rr]^-{E(BG)^{\odot}A} \ar[dd]_-{C+D(GB)^{\odot}GA} & & Z  \\
\\
U &    & W \ar[uu]_-{F+E(BG)^{\odot}BH}  \ar[ll]^-{D(GB)^{\odot}H}
}
\end{aligned}
\end{eqnarray}
The identity morphism of $(X,U)$ is 
$
\scriptsize{\begin{bmatrix}
X & 0 \\
0 & U
\end{bmatrix}}$.

The 2-morphisms $\alpha : R\Ra S : (X,U)\to (Y,V)$ are 2-morphisms $\alpha : R\Ra S : X+V\to Y+U$; and these can be written as $2\times 2$ matrices of 2-morphisms between the entries of the matrices for $R$ and $S$. The associativity equivalences are obtained using Lemma~\ref{6.1of51}. 

The autonomous monoidal structure on $\mathrm{i}\CM$ has the symmetric tensor product
given on objects by $$(X,U)\boxplus(X',U') = (X+X',U+U')$$ and the dual given by $(X,U)^*=(U,X)$
with counit as in diagram \eqref{autoncounit}. 
 \begin{eqnarray}\label{autoncounit}
\begin{aligned}
\xymatrix{
X+U \ar[rr]^-{} \ar[d]_-{\scriptsize{\begin{bmatrix}
0 & U \\
X & 0
\end{bmatrix}}} && 0  \\
U+X  && 0 \ar[u]_-{} \ar[ll]^-{}}
\end{aligned}
\end{eqnarray}
There is another monoidal structure on $\mathrm{i}\CM$ defined on objects by 
$$(X,U)\boxtimes (X',U') = (X\ot X' + U\ot U',X\ot U'+ U\ot X') \ .$$ The inclusion pseudofunctor
$\CM \to \mathrm{i}\CM$ taking $X$ to $(X,0)$ is strong
monoidal in that it takes $\ot$ to $\boxtimes$.

\appendix

\end{document}